\documentclass[11pt]{article}
\usepackage[T1]{fontenc}
\usepackage{lmodern}
\usepackage[utf8]{inputenc}
\usepackage[american]{babel}
\usepackage{csquotes}
\usepackage{amsmath,amssymb,amsfonts,amsthm,amscd,bm}
\usepackage{graphicx,indentfirst}
\usepackage{microtype}
\usepackage{tikz}
\usetikzlibrary{arrows.meta,positioning,calc,decorations.pathreplacing}
\usepackage[backend=biber,style=numeric,sorting=nyt,maxbibnames=99,giveninits=true,doi=true,url=false,isbn=false]{biblatex}
\AtEveryBibitem{\ifentrytype{online}{\clearfield{url}}{\iffieldundef{doi}{}{\clearfield{eprint}}}}
\usepackage[colorlinks=true,linkcolor=blue!50!black,citecolor=blue!50!black,urlcolor=blue!50!black]{hyperref}
\hypersetup{pdftitle={Stopping Rules and the Poisson Boundary of Random Walks on Countable Groups},pdfauthor={Yinghan Zhang},pdfsubject={Stopping distributions, entropy, and Poisson boundaries},pdfkeywords={Poisson boundary, stopping time, Green potential, entropy}}
\advance\textwidth by +1.0in \advance\textheight by +1.0in
\advance\oddsidemargin by -0.5in \advance\evensidemargin by -1.0in
\advance\topmargin by -0.5in
\numberwithin{equation}{section}
\newtheorem{theorem}{Theorem}[section]
\newtheorem{proposition}[theorem]{Proposition}
\newtheorem{corollary}[theorem]{Corollary}
\newtheorem{lemma}[theorem]{Lemma}
\theoremstyle{definition}

\newtheorem{question}[theorem]{Question}
\theoremstyle{remark}

\newtheorem*{TheoremA}{Theorem A}
\newtheorem*{TheoremB}{Theorem B}
\newtheorem*{TheoremC}{Theorem C}
\newtheorem*{TheoremD}{Theorem D}
\newcommand{\N}{\mathbb N}
\newcommand{\Z}{\mathbb Z}

\newcommand{\E}{\mathbb E}
\newcommand{\Prob}{\mathbb P}
\newcommand{\F}{\mathcal F}
\newcommand{\HH}{\mathcal H^{\infty}}
\newcommand{\supp}{\operatorname{supp}}
\newcommand{\Law}{\operatorname{Law}}
\newcommand{\TV}{d_{\mathrm{TV}}}
\newcommand{\one}{\mathbf 1}
\newcommand{\Poi}{\mathcal P}
\newcommand{\FSym}{\operatorname{FSym}}
\newcommand{\Sym}{\operatorname{Sym}}

\newcommand{\code}{\mathcal C}
\newcommand{\Bmu}{(B_\mu,\nu_\mu)}
\newcommand{\Beta}{(B_\eta,\nu_\eta)}

\title{\large\bf Stopping Rules and the Poisson Boundary of Random Walks on Countable Groups}
\author{
 Yinghan Zhang\thanks{College of Mathematics and Systems Science, Shandong University of Science and Technology, Qingdao, 266590, P.R. China. E-mail: \texttt{zhangyinghan007@126.com}.}}
\date{}
\begin{document}
\maketitle
\begin{abstract}
In this paper, we study the terminal laws and Poisson boundaries obtained by stopping random walks on countable groups. For a transient walk, we give a direct Green-potential characterization of randomized stopping distributions and show that the terminal law determines the mean stopping time. 
On free groups, this yields finitely supported randomized stopping distributions at a positive total variation distance from every ordinary stopping distribution. 
We construct symmetric, fully supported recurrent measures of finite entropy on countable locally finite groups and obtain ordinary strict boundary enlargements when the group is not Choquet--Deny. For generating source and terminal measures of finite entropy, positive asymptotic entropy of the source forces the stopping time to have finite mean. A direct proof of the finite-mean entropy formula and boundary preservation then shows that a finite-entropy stopping chain has at most one strict enlargement. An exact stopped-word entropy identity shows that infinite-mean ordinary stopping with a finite-entropy terminal law has infinite conditional history entropy.


\end{abstract}

\section{Introduction}\label{sec:intro}
 The Poisson boundary of a random walk represents its bounded harmonic functions and records the part of its asymptotic behavior that survives every finite shift of time. The entropy theory of Kaimanovich and Vershik \cite{KaimanovichVershik1983} provides a fundamental criterion for the triviality of this boundary. Kaimanovich's conditional entropy criterion \cite{Kaimanovich2000} gives a corresponding method for identifying a prescribed boundary with the full Poisson boundary. These methods have led to precise descriptions for broad classes of groups. In particular, Chawla, Forghani, Frisch, and Tiozzo \cite{ChawlaForghaniFrischTiozzo2025} identify the boundary of a non-elementary hyperbolic group for every generating measure of finite entropy. For wreath products, Frisch and Silva \cite{FrischSilva2026} identify the boundary with the joint data of the limiting lamp configuration and the boundary of the projected walk, under finite entropy and almost sure stabilization of the lamps.

Stopping times provide a natural way to change the step distribution while retaining a pathwise relation between the two walks. Let $G$ be a countable group, let $X_n=g_1\cdots g_n$ be the right random walk with step distribution $\mu$, and let $\eta=\Law(X_\tau)$ for an almost surely finite positive stopping time $\tau$. Optional stopping gives
\[
\HH(G,\mu)\subseteq\HH(G,\eta).
\]
Iterating the stopping rule realizes the $\eta$-walk as a subsequence of the $\mu$-walk. Its Poisson boundary therefore has the original boundary as a canonical factor. The reverse inclusion depends on the stopping rule.

Recent constructions show that this factor can be proper. Chawla and Frisch \cite{ChawlaFrisch2025} construct randomized stopping transformations with strictly larger spaces of bounded harmonic functions and obtain measures on free groups whose geometric boundary is a proper factor of the Poisson boundary. Their construction draws on the arboreal methods developed by Erschler and Kaimanovich \cite{ErschlerKaimanovich2023}. Forghani and Frisch \cite{ForghaniFrisch2026} show that convex combinations of Liouville measures can be non-Liouville, including finite-entropy examples on the lamplighter group over $\Z$. These results motivate the study of realizable terminal laws and of stopping transformations that strictly enlarge the boundary. General criteria for the inclusion of spaces of bounded harmonic functions are also studied by Hartman, Hru\v{s}kov\'{a}, and Segev \cite{HartmanHruskovaSegev2026}.

We studied the Dirichlet problem at infinity and Poisson boundary for certain manifolds without conjugate points in \cite{LiuZhang2025}. 
In this paper, we study stopping rules and the Poisson boundary of random walks on countable groups.
We study these problems through the potential of the source walk, the finite histories available to the stopping rule, and the entropy of the stopped walk. Randomized stopping allows the mass arriving at a history to be divided. Ordinary stopping, taken with respect to the natural filtration, must stop either the whole history cylinder or none of it. We show that this distinction can produce a uniform total variation gap, even when the randomized rule uses only times one and two. Recurrence has the opposite effect: it allows the walk to generate its own random target and then reach that target almost surely.

The Green-potential criterion below is a discrete realization of the potential-order principle in Rost's stopping-distribution theory \cite{Rost1971}. We include a direct proof by a state-dependent killing rule. The finite-mean entropy scaling formula is due to Forghani \cite{Forghani2017}, following earlier Abramov-type formulas for walks induced on finite-index subgroups \cite{HartmanLimaTamuz2014}. We give a proof that derives boundary preservation from entropy, together with a converse finite-mean criterion when the source has positive asymptotic entropy. The explicit separation estimates, the recurrent realization argument, and the stopped-word identity provide complementary information about the stopping rule itself.

We now state the main results. Throughout, a measure is called \emph{generating} if its support generates $G$ as a semigroup. We write $B_\mu\preccurlyeq B_\eta$ when there is a measure-preserving $G$-equivariant factor map from $B_\eta$ onto $B_\mu$, and $B_\mu\prec B_\eta$ when the canonical factor is proper. All logarithms are natural.

\begin{TheoremA}
Let $\mu$ be a transient probability measure on a countable group, and set
\[
U=\sum_{n\geq0}\mu^{*n}.
\]
A probability measure $\eta$ is the terminal law of an almost surely finite randomized stopping time $\tau\geq0$ if and only if $\eta*U\leq U$ pointwise. It is the terminal law of such a stopping time with $\tau\geq1$ if and only if
\[
\eta*U\leq U-\delta_e.
\]
Every embedding of $\eta$, ordinary or randomized, satisfies
\[
\E\tau=\sum_{x\in G}\bigl(U(x)-(\eta*U)(x)\bigr).
\]
\end{TheoremA}

\begin{TheoremB}
Let $\mu$ be the simple random walk measure on the free group $F_d$, where $d\geq2$, and put $q=2d$. If $2/q<p<1$, the measure
\[
\eta=p\mu+(1-p)\mu^{*2}
\]
is a randomized stopping distribution, whereas every ordinary positive-time stopping distribution $\rho$ satisfies
\[
\TV(\rho,\eta)\geq\min\left\{\frac{1-p}{q},\ p-\frac2q\right\}>0.
\]
Thus ordinary stopping distributions need not form a convex set and need not be dense among randomized stopping distributions.
\end{TheoremB}

The same conclusion holds for a range of lazy free-group walks, with explicit bounds given in Theorem~\ref{thm:lazy}. In each example the target is symmetric, finitely supported, and of finite entropy.

\begin{TheoremC}
Every probability measure on a countable group is an ordinary positive-time stopping distribution of any irreducible recurrent walk whose step distribution is not a point mass. Moreover, every countable infinite locally finite group admits a symmetric, fully supported recurrent measure of finite entropy. If this locally finite group is not Choquet--Deny, there exist symmetric, fully supported measures $\mu,\eta$ of finite entropy and an ordinary almost surely finite positive stopping time $\tau$ such that $\mu$ is recurrent and Liouville and
\[
\eta=\mu_\tau,\qquad B_\mu\prec B_\eta.
\]
\end{TheoremC}

Here a group is Choquet--Deny if every generating probability measure on it is Liouville. The existence of the non-Liouville target in Theorem~C follows from the theorem of Frisch, Hartman, Tamuz, and Vahidi Ferdowsi \cite{FrischHartmanTamuzVahidiFerdowsi2019}. The recurrent source and its ordinary stopping realization are constructed in Section~\ref{sec:recurrence}.

\begin{TheoremD}
Let $\mu$ and $\eta=\mu_\tau$ be generating probability measures on the same countable group, where $1\leq\tau<\infty$ almost surely, with randomization allowed.
\begin{enumerate}
\item If $H(\mu),H(\eta)<\infty$ and $h(\mu)>0$, then $\E\tau<\infty$.
\item If $H(\mu)<\infty$ and $m=\E\tau<\infty$, then
\[
H(\eta)<\infty,\qquad h(\eta)=m h(\mu),\qquad B_\eta\cong B_\mu
\]
through the canonical factor map.
\item If both measures have finite entropy and $B_\mu\prec B_\eta$, then
\[
h(\mu)=0,\qquad \E\tau=\infty,\qquad h(\eta)>0.
\]
\end{enumerate}
\end{TheoremD}

For ordinary stopping, the last conclusion has an exact information-theoretic form. If $W=(g_1,\ldots,g_\tau)$ is the stopped word and $0<H(\mu)<\infty$, then
\[
H(W)=H(\mu)\E\tau=H(X_\tau)+H(W\mid X_\tau).
\]
Consequently, a finite-entropy strict enlargement by ordinary stopping has $H(W\mid X_\tau)=\infty$: infinitely much historical information is lost when a stopped word is replaced by its group product.

The paper is organized as follows. Section~\ref{sec:setup} fixes the conventions and constructs the canonical boundary factor. Section~\ref{sec:potential} proves Theorem~A and applies it to convex mixtures. Section~\ref{sec:prefix} describes ordinary stopping by complete prefix codes and proves Theorem~B and its lazy extension. Section~\ref{sec:recurrence} proves Theorem~C and gives an explicit recurrent measure on the group of finitely supported permutations. Section~\ref{sec:entropy} proves Theorem~D and the stopped-word entropy identity. Section~\ref{sec:conclusions} records consequences for stopping chains and the remaining finite-generation problem.

\section{Basic setting and the canonical boundary factor}\label{sec:setup}
\subsection{Random walks and stopping rules}
Let $G$ be a countable discrete group with identity $e$. For nonnegative measures $\alpha,\beta$, convolution is defined by
\[
(\alpha*\beta)(x)=\sum_{y\in G}\alpha(y)\beta(y^{-1}x).
\]
We put $\mu^{*0}=\delta_e$ and $g\alpha=\delta_g*\alpha$. The right random walk started at $x$ is
\[
X_0=x,\qquad X_n=xg_1\cdots g_n,
\]
where the increments $g_j$ are independent with common law $\mu$. Its transition operator is
\[
P_\mu f(x)=\sum_{s\in G}\mu(s)f(xs).
\]
Unless a starting point is specified, $X_0=e$. We use the natural filtration $\F_n=\sigma(g_1,\ldots,g_n)$, with $\F_0$ trivial, and set
\[
\HH(G,\mu)=\{f\in\ell^\infty(G):P_\mu f=f\}.
\]
The walk is irreducible precisely when $\mu$ is generating in the semigroup sense.

An \emph{ordinary stopping time} is a stopping time for $(\F_n)$. A \emph{randomized stopping time} is a stopping time after adjoining auxiliary random variables independent of the entire increment sequence. One may take a sequence of independent uniform variables; revealing all auxiliary variables at time zero gives the same class of admissible rules. The future increments retain their original conditional law in this enlarged filtration. We write
\[
\mu_\tau(g)=\Prob(X_\tau=g).
\]
Except in the explicitly stated time-zero version of Theorem~\ref{thm:potential}, stopping times are positive and almost surely finite. Rules are expressed in terms of increments and are restarted relative to the current position. Every restart uses fresh auxiliary randomness.

For probability measures on $G$, our total variation convention is
\[
\TV(\alpha,\beta)=\sup_{A\subseteq G}|\alpha(A)-\beta(A)|
=\frac12\sum_{g\in G}|\alpha(g)-\beta(g)|.
\]
The entropy and asymptotic entropy are
\[
H(\mu)=-\sum_g\mu(g)\log\mu(g),\qquad
h(\mu)=\lim_{n\to\infty}\frac{H(\mu^{*n})}{n}
\quad\text{if }H(\mu)<\infty.
\]
We take $0\log0=0$. The limit exists by subadditivity, since the product of two independent group-valued variables has entropy at most the sum of their entropies.

\subsection{The Poisson boundary}
For a generating measure $\mu$, its Poisson boundary is a standard probability $G$-space $\Bmu$ such that the Poisson transform
\begin{equation}\label{eq:poisson}
\Poi_\mu F(g)=\int_{B_\mu}F(b)\,d(g\nu_\mu)(b)
\end{equation}
is an order-preserving isometric linear isomorphism from $L^\infty(B_\mu,\nu_\mu)$ onto $\HH(G,\mu)$. The measure $\nu_\mu$ is $\mu$-stationary:
\[
\nu_\mu=\sum_g\mu(g)g\nu_\mu.
\]
It is quasi-invariant under $G$. To see the latter assertion, stationarity implies $g\nu_\mu\ll\nu_\mu$ for every product $g$ of elements of $\supp\mu$; applying the same assertion to $g^{-1}$ gives equivalence of the two measures.

There is a boundary variable $b_\mu$ on the path space such that, for every $F\in L^\infty(B_\mu,\nu_\mu)$,
\begin{equation}\label{eq:mart-limit}
\Poi_\mu F(X_n)\longrightarrow F(b_\mu)\qquad\text{almost surely}.
\end{equation}
The boundary variable has law $\nu_\mu$ when the walk starts at $e$, and law $g\nu_\mu$ when it starts at $g$. Equivalently, the boundary is the measurable quotient of path space by events invariant under deleting a finite initial segment. A \emph{$\mu$-boundary} is a measure-preserving equivariant factor of $\Bmu$. The measure $\mu$ is \emph{Liouville} if $\HH(G,\mu)$ consists of constants, or equivalently if $B_\mu$ is a one-point space modulo null sets.

\begin{proposition}\label{prop:factor}
Let $\tau\geq1$ be an almost surely finite ordinary or randomized stopping time, and let $\eta=\mu_\tau$. Then
\[
\HH(G,\mu)\subseteq\HH(G,\eta).
\]
If both measures are generating, this inclusion is induced by a canonical measure-preserving $G$-equivariant map
\[
q:B_\eta\longrightarrow B_\mu.
\]
The maps are compatible with successive stopping transformations.
\end{proposition}
\begin{proof}
For $f\in\HH(G,\mu)$ and $x\in G$, the process $f(xX_n)$ is a bounded martingale, also in the enlarged filtration. Optional stopping at $\tau\wedge n$ and bounded convergence give
\[
f(x)=\lim_n\E f(xX_{\tau\wedge n})=\E f(xX_\tau)=P_\eta f(x).
\]

Iterate the rule and let $T_0=0<T_1<T_2<\cdots$ be the successive terminal times. The blocks of increments, with their auxiliary variables, are independent and identically distributed. In particular, $Y_k=X_{T_k}$ is an $\eta$-walk and $T_k\geq k$. Set
\[
J=\Poi_\eta^{-1}\Poi_\mu.
\]
Applying \eqref{eq:mart-limit} to the original walk and to the sampled walk gives, in this coupling,
\begin{equation}\label{eq:coupled-factor}
JF(b_\eta)=\lim_k\Poi_\mu F(X_{T_k})=F(b_\mu).
\end{equation}
For bounded $F_1,F_2$, use this identity for $F_1$, $F_2$, and $F_1F_2$. It follows that $J(F_1F_2)=JF_1\,JF_2$ almost everywhere. Moreover, $J1=1$, $J$ is positive, and
\[
\int JF\,d\nu_\eta=\Poi_\mu F(e)=\int F\,d\nu_\mu.
\]
If $0\leq F_n\uparrow F$ and the sequence is bounded, positivity and equality of integrals imply $JF_n\uparrow JF$ almost everywhere.

These properties determine a measurable point map. Choose refining finite measurable partitions of $B_\mu$ which generate its sigma-algebra and separate points in a standard Borel model. The indicators of their atoms are sent by $J$ to indicators, with disjointness, unions, and refinement preserved. Outside one null set, a point of $B_\eta$ thus determines a consistent sequence of atom labels. Its distribution is the coding distribution of $\nu_\mu$, because $J$ preserves integrals. The labels therefore belong almost surely to the Borel coding image of $B_\mu$. The measurable inverse of this coding defines $q$ and gives $JF=F\circ q$. Integral preservation gives $q_*\nu_\eta=\nu_\mu$.

The Poisson transforms intertwine the left $G$-actions, and hence so does $J$. A countable generating family and the countability of $G$ imply that $q$ is equivariant modulo a common null set. If a second stopping rule is applied, the three martingale limits agree along the nested terminal subsequences. Their pullbacks, and hence their factor maps, compose.
\end{proof}

\section{Green potentials and randomized stopping distributions}\label{sec:potential}
\subsection{A potential criterion}
Assume in this section that the $\mu$-walk is transient, so its Green measure
\[
U(x)=\sum_{n\geq0}\mu^{*n}(x)
\]
is finite at every $x\in G$. The total mass of $U$ is infinite. All inequalities between Green measures below are pointwise.

\begin{theorem}\label{thm:potential}
A probability measure $\eta$ is the law of $X_\tau$ for an almost surely finite randomized stopping time $\tau\geq0$ if and only if
\begin{equation}\label{eq:pot-zero}
\eta*U\leq U.
\end{equation}
Such a realization with $\tau\geq1$ exists if and only if
\begin{equation}\label{eq:pot-positive}
\eta*U\leq\mu*U=U-\delta_e.
\end{equation}
\end{theorem}
\begin{proof}
Suppose first that $\tau$ is given. The expected occupation measure before stopping is
\[
A(x)=\E\sum_{j=0}^{\tau-1}\one_{\{X_j=x\}}.
\]
Split the visits into those before $\tau$ and those at or after $\tau$. The strong Markov property and Tonelli's theorem give
\begin{equation}\label{eq:occupation}
U(x)=A(x)+\sum_y\eta(y)\sum_{n\geq0}\mu^{*n}(y^{-1}x)
=A(x)+(\eta*U)(x).
\end{equation}
This proves \eqref{eq:pot-zero}. If $\tau\geq1$, then $A\geq\delta_e$, which gives \eqref{eq:pot-positive}.

Now assume \eqref{eq:pot-zero}. Define the nonnegative measures
\[
A=U-\eta*U,\qquad M=A+\eta.
\]
These differences are finite at each coordinate. Since $U*\mu=U-\delta_e$, convolution gives
\[
A*\mu=A-\delta_e+\eta,\qquad M=\delta_e+A*\mu.
\]
Also $\eta*U\geq\eta$, and hence $M\leq U$. At every arrival at a state $x$ before stopping, kill the walk with probability
\[
q(x)=\frac{\eta(x)}{M(x)},\qquad
r(x)=1-q(x)=\frac{A(x)}{M(x)}
\quad(M(x)>0).
\]
Set $q(x)=1$ when $M(x)=0$. Independent coins are used at successive arrivals. Write
\[
P(x,y)=\mu(x^{-1}y),\qquad Q(x,y)=r(x)P(x,y).
\]
Then $MQ=A*\mu$, so $M=\delta_e+MQ$.

The actual total arrival measure before killing, with the killing arrival included, is
\[
v=\sum_{n\geq0}\delta_eQ^n.
\]
Iteration of the equation for $M$ shows that $v\leq M$ and $v=\delta_e+vQ$. Thus $D=M-v$ is nonnegative and satisfies $D=DQ$. For every $n$,
\[
0\leq D=DQ^n\leq UP^n=\sum_{j\geq n}\mu^{*j}.
\]
The last expression tends to zero at every coordinate by transience. Therefore $v=M$. The probability of being killed at $x$ is $v(x)q(x)=\eta(x)$. Summing over $x$ shows that killing occurs almost surely and has terminal law $\eta$.

Under \eqref{eq:pot-positive}, force the first step and apply the same construction with initial distribution $\mu$. Explicitly, put
\[
A_+=\mu*U-\eta*U,\qquad M_+=A_++\eta=\mu+A_+*\mu.
\]
Here $0\leq M_+\leq\mu*U\leq U$. With killing probability $\eta(x)/M_+(x)$, the arrival measure is $v_+=\sum_{n\geq0}\mu Q^n$. The same argument gives $v_+=M_+$ and terminal law $\eta$. The first possible killing time is one.
\end{proof}

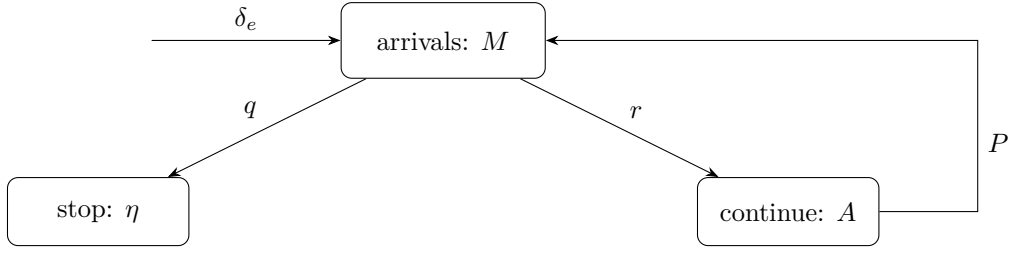
\begin{figure}[tb]
\centering
\begin{tikzpicture}[>=Stealth, every node/.style={font=\small},node distance=16mm]
\node[draw,rounded corners,minimum width=27mm,minimum height=10mm] (M) {arrivals: $M$};
\node[draw,rounded corners,below left=13mm and 20mm of M,minimum width=24mm,minimum height=9mm] (eta) {stop: $\eta$};
\node[draw,rounded corners,below right=13mm and 20mm of M,minimum width=24mm,minimum height=9mm] (A) {continue: $A$};
\draw[->] (M) -- node[above left] {$q$} (eta);
\draw[->] (M) -- node[above right] {$r$} (A);
\draw[->] (A.east) -- ++(13mm,0) |- node[pos=.2,right] {$P$} (M.east);
\draw[->] ([xshift=-25mm]M.west) -- node[above] {$\delta_e$} (M.west);
\node[below=9mm of $(eta.south)!.5!(A.south)$] {$M=\eta+A=\delta_e+A*\mu$};
\end{tikzpicture}
\caption{Mass balance for the killing construction. Transience makes the nonnegative remainder $M-v$ vanish, so the prescribed and actual arrival measures agree.}
\label{fig:killing}
\end{figure}

\begin{corollary}\label{cor:mean-potential}
Every ordinary or randomized embedding of $\eta$ in a transient $\mu$-walk satisfies the extended-value identity
\begin{equation}\label{eq:potential-mean}
\E\tau=\sum_{x\in G}\bigl(U(x)-(\eta*U)(x)\bigr).
\end{equation}
\end{corollary}
\begin{proof}
By \eqref{eq:occupation}, the summand is $\E\sum_{j<\tau}\one_{\{X_j=x\}}$. Tonelli's theorem gives
\[
\sum_x A(x)=\E\sum_{j<\tau}\sum_x\one_{\{X_j=x\}}=\E\tau.
\]
\end{proof}

The sum in \eqref{eq:potential-mean} is formed from finite nonnegative coordinate differences. In particular, the expected duration cannot be changed by replacing one embedding of a given target with another.

\subsection{Mixtures and a return-probability obstruction}
\begin{proposition}\label{prop:mixture}
Let $0<t<1$ and $\eta_t=t\mu+(1-t)\beta$. For a transient source $\mu$, the measure $\eta_t$ is a positive-time randomized stopping distribution if and only if $\beta$ is one.
\end{proposition}
\begin{proof}
Theorem~\ref{thm:potential} gives the equivalent inequalities
\[
t\mu*U+(1-t)\beta*U\leq\mu*U
\quad\Longleftrightarrow\quad
\beta*U\leq\mu*U.
\]
Since $(\mu*U)(x)$ is finite, the cancellation is valid at each coordinate; an infinite value of $\beta*U(x)$ violates both inequalities.
\end{proof}

Let $T_e$ be the first hitting time of $e$, allowing time zero, and let $T_e^+$ be the first positive return time. For a generating transient walk, set
\[
f(g)=\Prob_g(T_e<\infty),\qquad r_\mu=\Prob_e(T_e^+<\infty).
\]
The first-visit decomposition gives
\[
f(g)=\frac{U(g^{-1})}{U(e)},\qquad r_\mu=1-\frac1{U(e)}.
\]
Evaluating \eqref{eq:pot-positive} at $e$ yields the necessary condition
\begin{equation}\label{eq:return-test}
\sum_g\eta(g)f(g)\leq r_\mu,
\qquad\text{and hence}\qquad \eta(e)\leq r_\mu<1.
\end{equation}

The next application separates the existence of finite-entropy Liouville mixtures from their realization by stopping.

\begin{proposition}\label{prop:ff-obstruction}
On the lamplighter group $G=C_2\wr\Z$, there exist symmetric, fully supported Liouville measures $\alpha,\beta$ of finite entropy and $0<t<1$ such that
\[
\eta=t\alpha+(1-t)\beta
\]
is non-Liouville but is not a positive-time randomized stopping distribution of the $\alpha$-walk. It is not an ordinary stopping distribution even if time zero is allowed.
\end{proposition}
\begin{proof}
The finite-entropy theorem of Forghani and Frisch \cite[Theorem B]{ForghaniFrisch2026}, applied to $C_2\wr\Z$, supplies symmetric, fully supported finite-entropy Liouville measures $\lambda_0,\lambda_1$ such that every nontrivial convex mixture of them is non-Liouville.

The $\lambda_0$-walk is transient. To verify this, let $a$ toggle the lamp at the current position and let $s$ move the position one unit to the right. The Cayley graph for $\{a,s,s^{-1}\}$ contains a subdivided binary tree: at level $n$, either move right, or first toggle the current lamp and then move right. The lamps left behind record the binary choices. Intermediate toggle vertices have a lit lamp at the current position, whereas level vertices have not yet modified that position. Thus distinct tree vertices remain distinct in the Cayley graph, and each tree edge has length at most two. Sending equal flow into both branches gives finite energy, bounded by a constant times $\sum_{n\geq1}2^{-n}$. The finite-energy flow criterion \cite[Theorem 2.11]{LyonsPeres2016} proves transience of the simple walk. Since $\lambda_0$ is symmetric and positive on the three generators, its conductances dominate a positive multiple of these Cayley-graph conductances. The same flow proves its transience.

Put $r=r_{\lambda_0}<1$ and define
\[
\alpha=\lambda_0,\qquad
\beta=(1-\varepsilon)\delta_e+\varepsilon\lambda_1,
\qquad 0<\varepsilon<1.
\]
Adding a holding probability preserves bounded harmonic functions, since
\[
f=((1-\varepsilon)I+\varepsilon P_{\lambda_1})f
\quad\Longleftrightarrow\quad f=P_{\lambda_1}f.
\]
Symmetry, full support, and finite entropy are preserved as well. Writing $c=(1-t)(1-\varepsilon)$, we have
\[
\eta=c\delta_e+(1-c)\zeta,\qquad
\zeta=\frac{t}{t+(1-t)\varepsilon}\lambda_0+
\frac{(1-t)\varepsilon}{t+(1-t)\varepsilon}\lambda_1.
\]
Both coefficients of $\zeta$ are positive. Thus $\zeta$, and hence $\eta$, is non-Liouville. Choose $t,\varepsilon>0$ small enough that $\eta(e)\geq c>r$. This contradicts \eqref{eq:return-test} for a positive-time embedding. Finally, an ordinary rule can stop at time zero with positive probability only by stopping there almost surely, because $\F_0$ is trivial. Its terminal law would then be $\delta_e$, whereas $\eta\neq\delta_e$.
\end{proof}

\section{Ordinary stopping and prefix constraints}\label{sec:prefix}
\subsection{Complete stopping codes}
Let $A=\supp\mu$. For a nonempty finite word $w=(a_1,\ldots,a_n)$ over $A$, put
\[
p(w)=\prod_{j=1}^n\mu(a_j),\qquad \pi(w)=a_1\cdots a_n.
\]
A collection of finite words is \emph{prefix-free} if none of its elements is a proper prefix of another.

\begin{proposition}\label{prop:prefix}
A probability measure $\eta$ is an ordinary positive-time stopping distribution if and only if there is a prefix-free collection
$\code\subseteq\bigcup_{n\geq1}A^n$ such that
\begin{equation}\label{eq:prefix-law}
\sum_{w\in\code}p(w)=1,\qquad
\eta(g)=\sum_{\substack{w\in\code\\\pi(w)=g}}p(w)
\quad(g\in G).
\end{equation}
\end{proposition}
\begin{proof}
For an ordinary stopping time $\tau$, the event $\{\tau=n\}\in\F_n$ is a union of length-$n$ increment cylinders. Collect their words in $\code$. If a cylinder stops at time $n$, none of its extensions can be a first stopping cylinder. Hence $\code$ is prefix-free. Its cylinders are disjoint, each has probability $p(w)$, and their union has probability one exactly when $\tau<\infty$ almost surely. This proves \eqref{eq:prefix-law}. Completion of the filtration changes only null sets.

Conversely, stop when the observed word first belongs to $\code$. This is a stopping time. The first equation in \eqref{eq:prefix-law} gives almost sure finiteness, and the second gives the terminal law.
\end{proof}

For a randomized rule, let $a(w)$ be the probability of reaching a word $w$ without having stopped earlier, and let $t(w)$ be the mass stopped there. Then
\[
0\leq t(w)\leq a(w),\qquad a(ws)=(a(w)-t(w))\mu(s).
\]
Ordinary stopping imposes the additional condition $t(w)\in\{0,a(w)\}$. This constraint is retained even when the source has a holding probability.

\subsection{A uniform separation on free groups}
\begin{theorem}\label{thm:free}
Let $\mu=q^{-1}\sum_{s\in S}\delta_s$ be the simple random walk measure on $F_d$, where $d\geq2$, $S$ is the symmetric free generating set, and $q=2d$. For $2/q<p<1$, set $\eta=p\mu+(1-p)\mu^{*2}$. Then $\eta$ is a randomized stopping distribution, and every ordinary positive-time stopping distribution $\rho$ satisfies
\begin{equation}\label{eq:free-gap}
\TV(\rho,\eta)\geq\min\left\{\frac{1-p}{q},\ p-\frac2q\right\}.
\end{equation}
\end{theorem}
\begin{proof}
An independent coin choosing time one with probability $p$ and time two with probability $1-p$ realizes $\eta$.

Let $\rho=\mu_\tau$ for an ordinary rule. Every first-step cylinder $\{G_1=s\}$ is an atom of $\F_1$ of probability $1/q$. If the rule stops on one such cylinder, then $\rho(s)\geq1/q$. A word of length two in $S$ reduces to length zero or two, so $\eta(s)=p/q$. Therefore
\[
\TV(\rho,\eta)\geq\rho(s)-\eta(s)\geq\frac{1-p}{q}.
\]

If the rule stops on no first-step cylinder, then $\tau\geq2$ almost surely. Consequently,
\[
\rho(S)\leq\Prob_e(\exists n\geq2:X_n\in S).
\]
At time two, the walk is at $e$ with probability $1/q$ and at distance two with probability $(q-1)/q$. From distance two, the probability of ever reaching distance one is $1/(q-1)$. To obtain this value, let $u_R(r)$ be the probability for the radial chain, started at $r$, of hitting $1$ before $R$. The equations
\[
u_R(r)=q^{-1}u_R(r-1)+(q-1)q^{-1}u_R(r+1),\quad u_R(1)=1,\quad u_R(R)=0
\]
have the solution
\[
u_R(r)=\frac{(q-1)^{-(r-1)}-(q-1)^{-(R-1)}}{1-(q-1)^{-(R-1)}}.
\]
Letting $R\to\infty$ gives the stated hitting probability. A walk at $e$ reaches $S$ at the next step, so
\begin{equation}\label{eq:sphere-capacity}
\Prob_e(\exists n\geq2:X_n\in S)
=\frac1q+\frac{q-1}{q}\frac1{q-1}=\frac2q.
\end{equation}
Since $\eta(S)=p$, this case gives $\TV(\rho,\eta)\geq p-2/q$.
\end{proof}

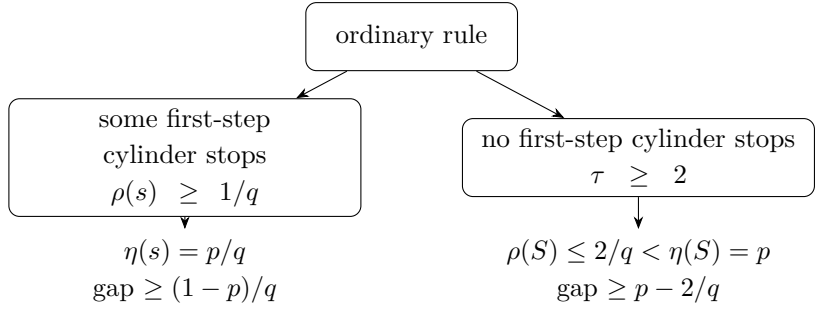
\begin{figure}[tb]
\centering
\begin{tikzpicture}[>=Stealth,every node/.style={font=\small}]
\node[draw,rounded corners,minimum width=28mm,minimum height=9mm] (root) at (0,0) {ordinary rule};
\node[draw,rounded corners,align=center,text width=44mm] (yes) at (-3,-1.6) {some first-step cylinder stops\\$\rho(s)\geq1/q$};
\node[draw,rounded corners,align=center,text width=44mm] (no) at (3,-1.6) {no first-step cylinder stops\\$\tau\geq2$};
\node[align=center] (yg) at (-3,-3.1) {$\eta(s)=p/q$\\gap $\geq(1-p)/q$};
\node[align=center] (ng) at (3,-3.1) {$\rho(S)\leq2/q<\eta(S)=p$\\gap $\geq p-2/q$};
\draw[->] (root) -- (yes); \draw[->] (root) -- (no);
\draw[->] (yes) -- (yg); \draw[->] (no) -- (ng);
\end{tikzpicture}
\caption{The two exhaustive cases in Theorem~\ref{thm:free}. The first uses an indivisible cylinder; the second uses the probability of returning to the first sphere after time one.}
\label{fig:gap}
\end{figure}

For $F_2$ and $p=3/4$, the bound is $1/16$. Both $\mu$ and $\mu^{*2}$ are ordinary stopping distributions, so \eqref{eq:free-gap} proves both nonconvexity and failure of density.

\subsection{The obstruction persists with holding}
\begin{theorem}\label{thm:lazy}
With $F_d,q,\mu$ as in Theorem~\ref{thm:free}, assume
\[
\frac1{q+1}<a<\frac12,\qquad \frac1{q-1}<p<1,
\qquad b=1-a.
\]
Put $\mu_a=a\delta_e+b\mu$ and $\eta=p\mu_a+(1-p)\mu_a^{*2}$. Then $\eta$ is a randomized stopping distribution. For every ordinary positive-time stopping distribution $\rho$ of the $\mu_a$-walk,
\begin{equation}\label{eq:lazy-gap}
\TV(\rho,\eta)\geq
\min\left\{(1-p)b\left(a-\frac bq\right),\
\frac{(1-p)b(1-2a)}q,\
\frac{b^2((q-1)p-1)}q\right\}>0.
\end{equation}
\end{theorem}
\begin{proof}
The same independent choice of times one and two realizes $\eta$. Direct convolution gives
\[
\mu_a^{*2}(e)=a^2+\frac{b^2}q,\qquad
\mu_a^{*2}(s)=\frac{2ab}q\quad(s\in S).
\]
Hence
\[
\mu_a(e)-\eta(e)=(1-p)b\left(a-\frac bq\right),\qquad
\mu_a(s)-\eta(s)=\frac{(1-p)b(1-2a)}q.
\]
If an ordinary rule stops on any first-step cylinder, its endpoint mass at that state is at least the entire cylinder mass. The displayed differences give the first two bounds in \eqref{eq:lazy-gap}.

Otherwise $\tau\geq2$. Let $B_1=\{e\}\cup S$. At time two the walk lies at distance two with probability $b^2(q-1)/q$, and otherwise lies in $B_1$. Deleting holding steps leaves simple random walk, so the probability of hitting $B_1$ from distance two remains $1/(q-1)$. Therefore
\[
\rho(B_1)\leq\Prob(\exists n\geq2:X_n\in B_1)
=1-\frac{b^2(q-2)}q.
\]
On the other hand,
\[
\eta(B_1)=1-\frac{(1-p)b^2(q-1)}q.
\]
Their difference is the third bound in \eqref{eq:lazy-gap}. The parameter assumptions make all three bounds positive.
\end{proof}

For $q=4$, $a=1/3$, and $p=1/2$, one has $\mu_a(e)=1/3$, $\mu_a(s)=1/6$, $\eta(e)=5/18$, and $\eta(s)=5/36$. The three lower bounds are $1/18$, $1/36$, and $1/18$, respectively. Thus $\TV(\rho,\eta)\geq1/36$ for every ordinary stopping law $\rho$.

The same first-step argument gives a general necessary condition. If a target satisfies $\eta(x)<\mu(x)$ for all $x\in\supp\mu$, an ordinary realization must have $\tau\geq2$. Consequently, for every $B\subseteq G$,
\begin{equation}\label{eq:first-capacity}
\eta(B)\leq\Prob_e(\exists n\geq2:X_n\in B)
=\sum_x\mu^{*2}(x)\Prob_x(T_B<\infty),
\end{equation}
where $T_B$ allows time zero. The estimates above combine the prefix constraint with this hitting-probability bound.

\section{Recurrence and ordinary realization}\label{sec:recurrence}
\subsection{Sampling a target and then hitting it}
\begin{theorem}\label{thm:recurrent}
Let $\mu$ be a non-Dirac probability measure whose random walk on the countable group $G$ is irreducible and recurrent. For every probability measure $\eta$ on $G$, there is an ordinary stopping time $1\leq\tau<\infty$ almost surely such that $\mu_\tau=\eta$.
\end{theorem}
\begin{proof}
Choose $A\subseteq G$ with $r=\mu(A)\in(0,1)$, and set $Z_n=\one_{\{G_n\in A\}}$. Read these indicators in successive pairs. Discard $00$ and $11$, and interpret $01$ as a zero bit and $10$ as a one bit. The retained bits are independent and fair. The number of pairs needed for each retained bit has a geometric distribution with success probability $2r(1-r)>0$, and is almost surely finite. This waiting time is independent of the value of the retained bit, because the two successful pairs have the same probability.

Enumerate the positive-mass states of $\eta$ as $y_1,y_2,\ldots$, and partition $[0,1)$ into consecutive half-open intervals $I_j$ of lengths $\eta(y_j)$. The infinite sequence of fair bits defines a uniform variable $V\in[0,1)$. After reading $k$ bits, its dyadic interval $J_k$ of length $2^{-k}$ is known. Stop reading as soon as $J_k\subseteq I_j$ for some $j$, and declare $Y=y_j$. Outside the countable set of interval endpoints, $V$ is in the interior of one $I_j$ and $|J_k|\to0$, so this declaration occurs after finitely many bits. The declaration time $\sigma$, measured in the original increment sequence, is therefore an ordinary almost surely finite stopping time. The variable $Y$ is $\F_\sigma$-measurable and has law $\eta$. For a point-mass target, take $\sigma=2$ and declare that target.

Define
\begin{equation}\label{eq:sample-hit}
\tau=\inf\{n\geq\sigma:X_n=Y\}.
\end{equation}
This is an ordinary stopping time, since
\[
\{\tau\leq n\}=
\bigcup_{k=1}^n\ \bigcup_{y\in G}
\{\sigma=k,\ Y=y,\ X_j=y\text{ for some }k\leq j\leq n\}
\in\F_n.
\]
Conditional on $\F_\sigma$, both $X_\sigma$ and $Y$ are fixed. Irreducibility and recurrence imply that the walk starting at $X_\sigma$ hits $Y$ almost surely. The strong Markov property therefore gives $\tau<\infty$ almost surely. By construction $X_\tau=Y$, and hence $\mu_\tau=\eta$.
\end{proof}

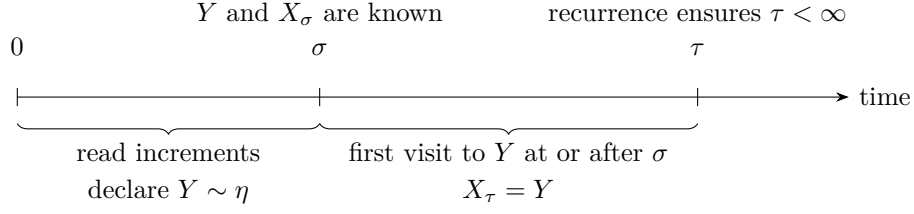
\begin{figure}[tb]
\centering
\begin{tikzpicture}[>=Stealth,every node/.style={font=\small}]
\draw[->] (0,0)--(11,0) node[right] {time};
\foreach \x/\lab in {0/0,4/\sigma,9/\tau}{\draw (\x,-.12)--(\x,.12) node[above=3mm] {$\lab$};}
\draw[decorate,decoration={brace,mirror,amplitude=4pt}] (0,-.35)--(4,-.35);
\draw[decorate,decoration={brace,mirror,amplitude=4pt}] (4,-.35)--(9,-.35);
\node[align=center] at (2,-1) {read increments\\declare $Y\sim\eta$};
\node[align=center] at (6.5,-1) {first visit to $Y$ at or after $\sigma$\\$X_\tau=Y$};
\node[align=center] at (4,1.1) {$Y$ and $X_\sigma$ are known};
\node[align=center] at (9,1.1) {recurrence ensures $\tau<\infty$};
\end{tikzpicture}
\caption{The ordinary realization in \eqref{eq:sample-hit}. The sampled target may depend on the position at time $\sigma$; the hitting argument conditions on both.}
\label{fig:recurrent}
\end{figure}

\begin{proposition}\label{prop:sampling-mean}
If $H(\eta)<\infty$, the sampling stage in Theorem~\ref{thm:recurrent} can be chosen with $\E\sigma<\infty$.
\end{proposition}
\begin{proof}
Let $K$ be the number of fair bits used by the interval procedure. Conditional on $V\in I_j$, failure to decide after $k$ bits implies that $V$ is within $2^{-k}$ of an endpoint of $I_j$. Thus
\[
\Prob(K>k\mid V\in I_j)\leq
\min\left\{1,\frac{2^{1-k}}{\eta(y_j)}\right\}.
\]
Writing $L_j=\log_2(1/\eta(y_j))$ and summing the tail before and after $\lceil L_j\rceil+1$ gives
\[
\E(K\mid V\in I_j)\leq L_j+4,
\qquad \E K\leq\frac{H(\eta)}{\log2}+4.
\]
The number of original increments used per retained bit has mean $1/[r(1-r)]$ and is independent of the retained bits. Therefore, for a non-point-mass target,
\[
\E\sigma\leq\frac{H(\eta)/\log2+4}{r(1-r)}<\infty.
\]
The chosen sampling time for a point-mass target is two.
\end{proof}

\subsection{Finite-entropy recurrent measures on locally finite groups}
A group is \emph{locally finite} if each of its finitely generated subgroups is finite.

\begin{proposition}\label{prop:locally-finite}
Every countable infinite locally finite group $L$ admits a symmetric, fully supported recurrent probability measure $\rho$ with $H(\rho)<\infty$.
\end{proposition}
\begin{proof}
Choose a strictly increasing exhaustion by finite subgroups,
\[
K_1<K_2<\cdots,\qquad \bigcup_nK_n=L,\qquad m_n=|K_n|\geq2.
\]
Let $u_n$ be the uniform measure on $K_n$. Define $r_0=1$ and, recursively,
\[
r_n=\min\left\{\frac{r_{n-1}}2,\frac1{2m_n},
\frac{2^{-n}}{1+\log m_{n+1}}\right\},\qquad p_n=r_{n-1}-r_n.
\]
The exhaustion is fixed in advance, so these numbers are well defined. Each $p_n$ is positive, $r_n\leq2^{-n}$, and $\sum_np_n=1$. Set
\[
\rho=\sum_{n\geq1}p_nu_n.
\]
This measure is symmetric and fully supported.

To estimate its entropy, first sample an index $N$ with law $(p_n)$ and then a uniform element of $K_N$. The chain rule gives
\begin{equation}\label{eq:mixture-entropy}
H(\rho)\leq H((p_n))+\sum_np_n\log m_n.
\end{equation}
For $n\geq2$,
\[
p_n\log m_n\leq r_{n-1}\log m_n\leq2^{-(n-1)}.
\]
Also $p_n\leq2^{-(n-1)}$. Since $-x\log x$ is increasing on $[0,e^{-1}]$, for $n\geq3$,
\[
-p_n\log p_n\leq(n-1)(\log2)2^{-(n-1)}.
\]
Both sums in \eqref{eq:mixture-entropy} are finite.

If $i\leq j$, multiplication of uniform subgroup measures gives
\begin{equation}\label{eq:nested-uniform}
u_i*u_j=u_j*u_i=u_j.
\end{equation}
Each $x\in K_j$ has exactly $|K_i|$ representations $x=yz$ with $y\in K_i$ and $z\in K_j$. Thus, conditional on indices $N_1,\ldots,N_k$, the product of $k$ sampled increments is uniform on $K_{\max_iN_i}$. Put $q_n=\sum_{j\leq n}p_j=1-r_n$ and $q_0=0$. It follows that
\[
\rho^{*k}(e)=\sum_{n\geq1}\frac{q_n^k-q_{n-1}^k}{m_n}.
\]
All summands are nonnegative. Tonelli's theorem and the geometric-series formula now give
\begin{align*}
\sum_{k\geq1}\rho^{*k}(e)
&=\sum_{n\geq1}\frac1{m_n}
\left(\frac{q_n}{1-q_n}-\frac{q_{n-1}}{1-q_{n-1}}\right)\\
&=\sum_{n\geq1}\frac1{m_n}\left(\frac1{r_n}-\frac1{r_{n-1}}\right).
\end{align*}
Since $r_n\leq r_{n-1}/2$ and $r_n\leq1/(2m_n)$, each summand in the last series is at least one. The return Green sum diverges, proving recurrence of the irreducible walk.
\end{proof}

\begin{lemma}\label{lem:rec-liouville}
An irreducible recurrent random walk is Liouville.
\end{lemma}
\begin{proof}
For $x,y\in G$, recurrence gives $\Prob_x(T_y<\infty)=1$. If $f$ is bounded and harmonic, optional stopping at $T_y\wedge n$ and bounded convergence imply
\[
f(x)=\lim_n\E_x f(X_{T_y\wedge n})=f(y).
\]
\end{proof}

\begin{theorem}\label{thm:locally-finite-growth}
Let $L$ be a countable locally finite group which is not Choquet--Deny. There exist symmetric, fully supported finite-entropy measures $\mu,\eta$ on $L$ and an ordinary almost surely finite positive stopping time $\tau$ such that $\eta=\mu_\tau$, $\mu$ is recurrent and Liouville, and $\eta$ is non-Liouville.
\end{theorem}
\begin{proof}
A finite group is Choquet--Deny, so $L$ is infinite. Proposition~\ref{prop:locally-finite} supplies $\mu$, and Lemma~\ref{lem:rec-liouville} gives its Liouville property. The classification theorem of Frisch, Hartman, Tamuz, and Vahidi Ferdowsi \cite[Theorem 1]{FrischHartmanTamuzVahidiFerdowsi2019} supplies a symmetric, fully supported finite-entropy non-Liouville measure $\eta$. Apply Theorem~\ref{thm:recurrent} to realize this target by an ordinary stopping time. Proposition~\ref{prop:factor} then gives a proper canonical boundary factor.
\end{proof}

\subsection{An explicit source on the finite-permutation group}
Let $L=\FSym(\Z)$, the group of permutations of $\Z$ with finite support. Enumerate $\Z=\{z_1,z_2,\ldots\}$ and put
\[
K_n=\Sym\{z_1,\ldots,z_n\},\qquad
p_n=\frac1{(n-1)!}-\frac1{n!}=\frac{n-1}{n!}\quad(n\geq2).
\]
The measure
\begin{equation}\label{eq:explicit-source}
\rho=\sum_{n\geq2}p_nu_{K_n}
\end{equation}
is symmetric and fully supported. It has finite entropy because
\[
H(\rho)\leq H((p_n))+\sum_{n\geq2}p_n\log(n!)
\leq2\sum_{n\geq2}\frac{n\log n}{(n-1)!}<\infty.
\]
Here $-\log p_n\leq\log(n!)$ was used for the index entropy.

For $q_m=\Prob(N\leq m)=1-1/m!$ and $q_1=0$, identity \eqref{eq:nested-uniform} gives
\[
\rho^{*k}(e)=\sum_{m\geq2}\frac{q_m^k-q_{m-1}^k}{m!}.
\]
Consequently,
\begin{equation}\label{eq:factorial-return}
\sum_{k\geq1}\rho^{*k}(e)
=\sum_{m\geq2}\frac{m!-(m-1)!}{m!}
=\sum_{m\geq2}\left(1-\frac1m\right)=\infty.
\end{equation}
This proves recurrence directly.

Every nonidentity element of $L$ has infinitely many conjugates: move its nonempty finite support to pairwise disjoint finite subsets of $\Z$. Thus $L$ has the infinite conjugacy class property and is not Choquet--Deny by \cite{FrischHartmanTamuzVahidiFerdowsi2019}. Theorem~\ref{thm:locally-finite-growth} applies to the explicit source \eqref{eq:explicit-source}.

One can make the sampling bits visible in a separate coordinate. Let $u$ be the uniform measure on $C_2$, let $H=L\times C_2$, and set $\mu=\rho\times u$. Then
\[
H(\mu)=H(\rho)+\log2,\qquad
\mu^{*k}(e,e)=\frac12\rho^{*k}(e)\quad(k\geq1).
\]
Hence $\mu$ is recurrent and Liouville. Choose a symmetric, fully supported finite-entropy non-Liouville measure $\theta_0$ on $L$ and put $\eta=\theta_0\times u$. A nonconstant bounded $\theta_0$-harmonic function lifts to an $\eta$-harmonic function, so $\eta$ is non-Liouville. The $C_2$ increments are fair bits for the sampling stage of Theorem~\ref{thm:recurrent}, which realizes $\eta$ by ordinary stopping.

For comparison with finitely generated groups, let $s$ act on $\Z$ by translation and form
\[
\Gamma=(L\rtimes\langle s\rangle)\times C_2.
\]
Translation, the transposition $(0\ 1)$, and the nonidentity element of $C_2$ generate $\Gamma$: conjugating $(0\ 1)$ by powers of $s$ gives all adjacent transpositions. The subgroup $H$ is normal of infinite index, with $\Gamma/H\cong\Z$. Both measures constructed above have support generating $H$. Viewed as walks on $\Gamma$, they stay in the left cosets $gH$; their source is Liouville on each such coset, while their target has nonconstant bounded harmonic functions there. The construction thus concerns an infinitely generated support group inside a three-generated ambient group.

\section{Entropy and rigidity of stopping transformations}\label{sec:entropy}
\subsection{Boundary entropy and its pathwise limit}
Let $\alpha$ be generating, and let $(C,\lambda)$ be an $\alpha$-stationary probability $G$-space. Write
\[
K_g(b)=\frac{d(g\lambda)}{d\lambda}(b),\qquad c(g,b)=\log K_g(b).
\]
Stationarity and generation imply quasi-invariance, as in Section~\ref{sec:setup}. Since $G$ is countable, the kernels can be chosen on a common conull set so that the cocycle identity holds:
\begin{equation}\label{eq:cocycle}
c(gh,b)=c(g,b)+c(h,g^{-1}b).
\end{equation}
The Furstenberg entropy is
\begin{equation}\label{eq:furstenberg}
h_\alpha(C,\lambda)=\sum_g\alpha(g)\int K_g\log K_g\,d\lambda.
\end{equation}
For $H(\alpha)<\infty$, the classical entropy theorem states that
\begin{equation}\label{eq:classical-entropy}
h_\alpha(B_\alpha,\nu_\alpha)=h(\alpha),\qquad
h(\alpha)=0\ \Longleftrightarrow\ \alpha\text{ is Liouville};
\end{equation}
see \cite{KaimanovichVershik1983,Derriennic1980,Kaimanovich2000}.

\begin{lemma}\label{lem:cocycle-entropy}
Suppose $H(\alpha)<\infty$ and $(C,\lambda)$ is $\alpha$-stationary. For $n\geq1$, put
\[
J_n=\sum_g\alpha^{*n}(g)\int K_g\log K_g\,d\lambda.
\]
Then $0\leq J_n\leq H(\alpha^{*n})$, $J_{n+m}=J_n+J_m$, and
\begin{equation}\label{eq:Jn}
J_n=n h_\alpha(C,\lambda),\qquad h_\alpha(C,\lambda)\leq h(\alpha).
\end{equation}
If $(Z_n,b)$ has conditional law $\Law(b\mid Z_n=g)=g\lambda$ and $Z_n\sim\alpha^{*n}$, then
\begin{equation}\label{eq:negative-part}
\E c(Z_n,b)=J_n,\qquad \E[c(Z_n,b)^-]\leq\frac1e.
\end{equation}
\end{lemma}
\begin{proof}
Stationarity gives
\[
\sum_g\alpha^{*n}(g)K_g(b)=1
\]
almost everywhere. Thus $K_g\leq1/\alpha^{*n}(g)$ whenever $\alpha^{*n}(g)>0$. Since $\int K_g\,d\lambda=1$, convexity of $x\log x$ gives
\[
0\leq\int K_g\log K_g\,d\lambda\leq-\log\alpha^{*n}(g).
\]
This proves the bounds on $J_n$. Also,
\[
\sum_g\alpha^{*n}(g)\int_{\{K_g<1\}}-K_g\log K_g\,d\lambda
\leq\frac1e,
\]
which gives \eqref{eq:negative-part} and ensures integrability of the cocycles in the following calculation.

Integrate \eqref{eq:cocycle} against $gh\lambda$ and average over $g$ and $h$ with distributions $\alpha^{*n}$ and $\alpha^{*m}$. The term $c(h,g^{-1}b)$ integrates to $\int c(h,b)\,d(h\lambda)(b)$. For the term $c(g,b)$, summing over $h$ replaces $gh\lambda$ by $g\lambda$, since $\alpha^{*m}*\lambda=\lambda$. Summing over $g$ therefore gives $J_{n+m}=J_n+J_m$. It follows that $J_n=nJ_1$. Divide $J_n\leq H(\alpha^{*n})$ by $n$ and take the limit to obtain \eqref{eq:Jn}.
\end{proof}

\begin{lemma}\label{lem:boundary-shannon}
Let $\mu$ be generating with $H(\mu)<\infty$, and let $b_\mu$ be the boundary variable of the $\mu$-walk. For the kernels of $\Bmu$,
\begin{equation}\label{eq:boundary-shannon}
\frac1n c(X_n,b_\mu)\longrightarrow h(\mu)
\qquad\text{almost surely}.
\end{equation}
\end{lemma}
\begin{proof}
On the increment space $G^{\N}$, let $\vartheta$ be the left shift. Equivariance of the boundary variable gives
\[
b_\mu(\omega)=G_1(\omega)b_\mu(\vartheta\omega).
\]
Put $F(\omega)=c(G_1(\omega),b_\mu(\omega))$. Conditional on $G_1=g$, the boundary variable has law $g\nu_\mu$. Lemma~\ref{lem:cocycle-entropy} gives
\[
\E F^-\leq1/e,\qquad \E F^+\leq H(\mu),\qquad
\E F=h_\mu(B_\mu,\nu_\mu)=h(\mu).
\]
In particular, $F$ is integrable. Iterating \eqref{eq:cocycle} and the boundary identity gives
\[
c(X_n,b_\mu)=\sum_{j=0}^{n-1}F\circ\vartheta^j.
\]
The Bernoulli shift is ergodic, so Birkhoff's theorem proves \eqref{eq:boundary-shannon}.
\end{proof}

\subsection{Positive asymptotic entropy forces finite mean}
\begin{theorem}\label{thm:positive-mean}
Let $\mu$ and $\eta=\mu_\tau$ be generating measures of finite entropy, where $\tau$ is a positive almost surely finite ordinary or randomized stopping time. If $h(\mu)>0$, then $\E\tau<\infty$.
\end{theorem}
\begin{proof}
Iterate the rule as in Proposition~\ref{prop:factor}, obtaining $T_k$ and $Y_k=X_{T_k}$. The block lengths $T_j-T_{j-1}$ are independent and each has the same distribution as $\tau$. If $\E\tau=\infty$, then, for every positive integer $R$, the strong law applied to truncated block lengths gives
\[
\liminf_{k\to\infty}\frac{T_k}{k}\geq\E(\tau\wedge R).
\]
Letting $R\to\infty$ yields $T_k/k\to\infty$ almost surely. Lemma~\ref{lem:boundary-shannon}, applied along this subsequence, then gives
\begin{equation}\label{eq:divergent-cocycle}
\frac{c(Y_k,b_\mu)}k
=\frac{c(X_{T_k},b_\mu)}{T_k}\frac{T_k}{k}
\longrightarrow\infty.
\end{equation}

The original boundary measure $\nu_\mu$ is $\eta$-stationary. For a bounded measurable $F$, optional stopping of its Poisson transform gives
\[
\int F\,d\nu_\mu=\sum_g\eta(g)\int F\,d(g\nu_\mu).
\]
Moreover, conditional on the stopped history up to $T_k$, the remaining original increments have law $\mu^{\otimes\N}$. Hence the conditional law of $b_\mu$ given $Y_k=g$ is $g\nu_\mu$. Applying Lemma~\ref{lem:cocycle-entropy} to $\alpha=\eta$ on this boundary yields
\[
\E c(Y_k,b_\mu)=k h_\eta(B_\mu,\nu_\mu),\qquad
\E[c(Y_k,b_\mu)^-]\leq1/e,
\]
where $h_\eta(B_\mu,\nu_\mu)\leq H(\eta)<\infty$. Fatou's lemma applied to the positive parts in \eqref{eq:divergent-cocycle} gives
\[
\infty\leq\liminf_k\frac{\E[c(Y_k,b_\mu)^+]}k
=h_\eta(B_\mu,\nu_\mu)<\infty.
\]
This contradiction proves the assertion.
\end{proof}

\subsection{Finite-mean entropy scaling and boundary preservation}
We first record an entropy estimate for a segment between stopping times.

\begin{lemma}\label{lem:word-entropy-bound}
Let $H(\mu)<\infty$, let $S\leq R$ be finite stopping times in a filtration with independent auxiliary randomness, and assume $\ell=\E(R-S)<\infty$. For
\[
L=R-S,\qquad V=(G_{S+1},\ldots,G_R),
\]
including the empty word when $L=0$, one has
\begin{equation}\label{eq:word-bound}
H(V\mid L)\leq H(\mu)\ell,\qquad
H(L)\leq(\ell+1)\log(\ell+1)-\ell\log\ell
\leq\log(\ell+1)+1.
\end{equation}
\end{lemma}
\begin{proof}
Conditional on $L=l$, compare the law of $V$ with the product measure $\mu^{\otimes l}$. Nonnegativity of relative entropy bounds its entropy by its cross-entropy. Averaging over $l$ gives
\[
H(V\mid L)\leq\E\sum_{j=S+1}^R-\log\mu(G_j).
\]
The event $\{S<j\leq R\}=\{S\leq j-1\}\cap\{R\geq j\}$ is measurable before the $j$th increment. It is therefore independent of $G_j$. Tonelli's theorem yields
\[
\E\sum_{j=S+1}^R-\log\mu(G_j)
=H(\mu)\sum_{j\geq1}\Prob(S<j\leq R)=H(\mu)\ell.
\]
Among distributions on $\{0,1,\ldots\}$ with mean $\ell$, comparison with the geometric law of that mean gives
\[
H(L)\leq(\ell+1)\log(\ell+1)-\ell\log\ell.
\]
For $\ell>0$, the last expression is $\log(\ell+1)+\ell\log(1+1/\ell)\leq\log(\ell+1)+1$. The case $\ell=0$ is immediate.
\end{proof}

\begin{theorem}\label{thm:finite-mean}
Let $H(\mu)<\infty$ and $\eta=\mu_\tau$, where $1\leq\tau<\infty$ almost surely, randomization is allowed, and $m=\E\tau<\infty$. If both measures are generating, then
\begin{equation}\label{eq:abramov}
H(\eta)<\infty,\qquad h(\eta)=m h(\mu),
\end{equation}
and the canonical factor $q:B_\eta\to B_\mu$ is an isomorphism modulo null sets.
\end{theorem}
\begin{proof}
Apply Lemma~\ref{lem:word-entropy-bound} with $S=0$ and $R=\tau$. The terminal position is a function of the stopped word, so
\begin{equation}\label{eq:one-step-entropy}
H(\eta)\leq H(\tau)+mH(\mu)<\infty.
\end{equation}

We first compare asymptotic entropies. For the iterated terminal times $T_k$, put $N_k=\lfloor mk\rfloor$ and $L_k=|T_k-N_k|$. Knowledge of $X_{N_k}$, the order of $T_k$ and $N_k$, and the increment word between those times determines $X_{T_k}$. If $T_k<N_k$, right multiplication by the inverse of that word's product recovers $X_{T_k}$. Both $\min(T_k,N_k)$ and $\max(T_k,N_k)$ are stopping times. Lemma~\ref{lem:word-entropy-bound} therefore gives
\begin{equation}\label{eq:time-error}
H(\eta^{*k})\leq H(\mu^{*N_k})+\log2+H(L_k)+H(\mu)\E L_k.
\end{equation}
The $L^1$ law of large numbers for the independent integrable block lengths implies
\[
\frac1k\E|T_k-mk|\longrightarrow0,
\qquad \E L_k=o(k).
\]
It follows from \eqref{eq:word-bound} that $H(L_k)=o(k)$. Dividing \eqref{eq:time-error} by $k$ yields
\begin{equation}\label{eq:entropy-upper}
h(\eta)\leq m h(\mu).
\end{equation}

For the reverse inequality, use the original boundary $\Bmu$, its cocycle $c$, and the same coupling. By the strong law, $T_k/k\to m$. Lemma~\ref{lem:boundary-shannon} gives
\[
\frac{c(Y_k,b_\mu)}k\longrightarrow m h(\mu).
\]
The stationarity and conditional-law argument in Theorem~\ref{thm:positive-mean} applies unchanged. Thus Lemma~\ref{lem:cocycle-entropy} and Fatou's lemma give
\[
m h(\mu)\leq\liminf_k\frac{\E[c(Y_k,b_\mu)^+]}k
=h_\eta(B_\mu,\nu_\mu)\leq h(\eta).
\]
Together with \eqref{eq:entropy-upper}, this proves
\begin{equation}\label{eq:entropy-equalities}
m h(\mu)=h_\eta(B_\mu,\nu_\mu)=h(\eta).
\end{equation}

It remains to identify the factor. Write $(\widehat B,\widehat\nu)=\Beta$ and let $\mathcal A=q^{-1}\mathcal B(B_\mu)$, completed modulo $\widehat\nu$-null sets. For $g\in G$, put
\[
\widehat K_g=\frac{d(g\widehat\nu)}{d\widehat\nu},\qquad
k_g=\E_{\widehat\nu}(\widehat K_g\mid\mathcal A).
\]
The function $k_g$ is the pullback of the Radon--Nikodym kernel on the factor. By \eqref{eq:classical-entropy} for $\eta$ and \eqref{eq:entropy-equalities},
\begin{equation}\label{eq:jensen-equality}
0=\sum_g\eta(g)\left[
\int\widehat K_g\log\widehat K_g\,d\widehat\nu
-\int k_g\log k_g\,d\widehat\nu\right].
\end{equation}
Conditional Jensen's inequality makes each bracket nonnegative. For $\eta(g)>0$, its vanishing and strict convexity of $x\log x$ imply $\widehat K_g=k_g$ almost everywhere. Thus $\widehat K_g$ is $\mathcal A$-measurable for each $g\in\supp\eta$. Since $\mathcal A$ is $G$-invariant, the cocycle identity extends this measurability to products of support elements and hence to every $g\in G$.

If $F\in L^\infty(\widehat B,\widehat\nu)$ satisfies $\E(F\mid\mathcal A)=0$, then
\[
\Poi_\eta F(g)=\int F\widehat K_g\,d\widehat\nu=0
\quad(g\in G).
\]
Injectivity of the Poisson transform implies $F=0$. Apply this to the difference between any bounded function and its conditional expectation onto $\mathcal A$. Every bounded function is $\mathcal A$-measurable, so $\mathcal A$ is the full boundary sigma-algebra. The factor $q$ is an isomorphism.
\end{proof}

The entropy formula in \eqref{eq:abramov} is the stopping-time Abramov formula of \cite{Forghani2017}. The proof above obtains it together with boundary preservation. Estimate \eqref{eq:one-step-entropy} includes the information in a randomized duration. This term can be necessary: if $\mu=\delta_a$, $a\neq a^2$, and an independent fair coin chooses $\tau=1$ or $2$, then $H(\mu)=0$ while $H(\mu_\tau)=\log2$.

\begin{corollary}\label{cor:rigidity}
Suppose $\mu$ and $\eta=\mu_\tau$ generate the same group and have finite entropy. If the canonical factor $B_\eta\to B_\mu$ is proper, then
\[
h(\mu)=0,\qquad \E\tau=\infty,\qquad h(\eta)>0.
\]
\end{corollary}
\begin{proof}
If $h(\mu)>0$, Theorem~\ref{thm:positive-mean} gives finite mean, and Theorem~\ref{thm:finite-mean} makes the factor an isomorphism. Therefore $h(\mu)=0$, so the source is Liouville by \eqref{eq:classical-entropy}. Properness makes the target boundary nontrivial, whence $h(\eta)>0$. Finite mean is excluded once more by Theorem~\ref{thm:finite-mean}.
\end{proof}

\subsection{The entropy of an ordinary stopped history}
\begin{theorem}\label{thm:history}
Let $\tau$ be an ordinary almost surely finite positive stopping time, and let $W=(G_1,\ldots,G_\tau)$. If $0<H(\mu)<\infty$, then
\begin{equation}\label{eq:history-identity}
H(W)=H(\mu)\E\tau=H(X_\tau)+H(W\mid X_\tau)
\end{equation}
as an identity in $[0,\infty]$. In particular, if $H(\mu_\tau)<\infty$ and $\E\tau=\infty$, then $H(W\mid X_\tau)=\infty$.
\end{theorem}
\begin{proof}
Let $\code$ be the complete prefix code from Proposition~\ref{prop:prefix}. For $w\in\code$, the probability of the stopped word $w$ is exactly $p(w)$. Hence
\begin{align*}
H(W)
&=\sum_{w\in\code}p(w)(-\log p(w))\\
&=\E\sum_{j=1}^{\tau}-\log\mu(G_j)\\
&=\sum_{j\geq1}\Prob(\tau\geq j)H(\mu)
=H(\mu)\E\tau.
\end{align*}
Tonelli's theorem applies because all summands are nonnegative. The third equality uses $\{\tau\geq j\}\in\F_{j-1}$ and independence of $G_j$ from that event. Since $X_\tau=\pi(W)$ is a function of $W$, the discrete entropy chain rule, also valid as an extended nonnegative sum, gives
\[
H(W)=H(X_\tau)+\sum_g\Prob(X_\tau=g)H(W\mid X_\tau=g).
\]
This proves \eqref{eq:history-identity} and its final assertion.
\end{proof}

\begin{corollary}\label{cor:injective-code}
Suppose $0<H(\mu)<\infty$, the stopping time is ordinary, and $H(\mu_\tau)<\infty$. If the endpoint map $\pi$ is injective on the stopping code, then $\E\tau<\infty$. If both measures are generating, the canonical boundary factor is an isomorphism.
\end{corollary}
\begin{proof}
Injectivity gives $H(W\mid X_\tau)=0$. Theorem~\ref{thm:history} yields $H(\mu_\tau)=H(\mu)\E\tau$, so the mean is finite. Apply Theorem~\ref{thm:finite-mean}.
\end{proof}

For the strict enlargements in Theorem~\ref{thm:locally-finite-growth}, Corollary~\ref{cor:rigidity} gives $\E\tau=\infty$, and Theorem~\ref{thm:history} gives $H(W\mid X_\tau)=\infty$. Proposition~\ref{prop:sampling-mean} makes the sampling stage integrable. Thus, for the realization \eqref{eq:sample-hit},
\[
\E\sigma<\infty,\qquad \E(\tau-\sigma)=\infty.
\]
The infinite expected duration occurs in the return to the sampled target. The entropy identity records the resulting multiplicity of histories with the same endpoint.

\section{Consequences and further questions}\label{sec:conclusions}
The preceding results constrain both the terminal law and the boundary of a stopped walk. Finite-entropy rigidity has a particularly simple consequence for successive stopping transformations.

\begin{corollary}
Let $(\mu_n)_{n\geq0}$ be generating finite-entropy measures on one countable group. Suppose each $\mu_{n+1}$ is obtained from $\mu_n$ by an almost surely finite positive stopping rule, ordinary or randomized. Among the canonical factors
\[
B_{\mu_0}\preccurlyeq B_{\mu_1}\preccurlyeq B_{\mu_2}\preccurlyeq\cdots,
\]
at most one is proper. After a proper factor occurs, every subsequent stopping rule has finite mean and preserves the boundary.
\end{corollary}
\begin{proof}
By Corollary~\ref{cor:rigidity}, the first proper factor has a Liouville source and a target with positive asymptotic entropy. Theorems~\ref{thm:positive-mean} and \ref{thm:finite-mean} apply to every subsequent step and preserve its nontrivial boundary. Induction proves the assertion.
\end{proof}

The same argument applies to an ordered family when each later stage is an actual stopping distribution of every earlier stage and the canonical factors are compatible. A limit of step distributions alone supplies no such factor map; the stopping relation remains part of the hypothesis.

There is also a direct entropy consequence for transient sources. If $H(\mu)<\infty$, $h(\mu)>0$, and $\eta$ is a generating finite-entropy randomized stopping distribution, then
\[
\sum_x\bigl(U(x)-(\eta*U)(x)\bigr)<\infty.
\]
This follows from Theorem~\ref{thm:positive-mean} and Corollary~\ref{cor:mean-potential}. Thus an infinite Green deficit rules out finite entropy of any generating stopping target when the source has positive asymptotic entropy.

The ordinary strict enlargements constructed here occur on infinitely generated locally finite groups. They lead to the following problem.

\begin{question}\label{q:finite-generation}
Does there exist an infinite finitely generated group $G$, generating measures $\mu,\eta$ on $G$ with $H(\mu),H(\eta)<\infty$, and an ordinary almost surely finite positive stopping time $\tau$ such that $\eta=\mu_\tau$ and $B_\mu\prec B_\eta$?
\end{question}

Any such construction must have a Liouville source, a non-Liouville target, and an infinite-mean stopping rule. It must also satisfy $H(W\mid X_\tau)=\infty$. For a transient source, the target must obey $\eta*U\leq U-\delta_e$ as well as the prefix condition. Finite-entropy Liouville mixtures such as those in \cite{ForghaniFrisch2026} provide candidate pairs of endpoint measures. Proposition~\ref{prop:mixture} shows that moving a mixture closer to its source does not change the randomized embedding condition, while Proposition~\ref{prop:ff-obstruction} gives pairs for which that condition fails.

\section*{\textbf{Acknowledgements}}
Yinghan Zhang is partially supported by Natural Science Foundation of China under Grant No.~12101370
and Natural Science Foundation of Shandong Province under Grant ZR2020QA020.

\clearpage
\printbibliography[heading=bibintoc,title={References}]
\end{document}